\documentclass[11pt]{article}
 
\usepackage{ytableau}

\usepackage{amssymb,latexsym,amsmath,amsthm}
\usepackage{ytableau}
\usepackage[square,comma,numbers]{natbib}			
\usepackage[colorlinks = true, linkcolor = blue,citecolor = red]{hyperref} 	     
\usepackage{cleveref} 
\allowdisplaybreaks

\newtheorem{theo}{Theorem}

\newtheorem{lemm}[theo]{Lemma}
\newtheorem{coro}[theo]{Corollary}

\newtheorem{prop}[theo]{Proposition}
\numberwithin{theo}{section} 
\numberwithin{equation}{section}

\newcommand{\C}{\mathbb{C}}
\newcommand{\Z}{\mathbb{Z}}
\newcommand{\N}{\mathbb{N}}
\newcommand{\Span}{\operatorname{span}}
\newcommand{\Char}{\operatorname{char}}

\begin{document}

\begin{center}
{\Large \bf
A Lie algebra of Grassmannian Dirac operators\\[1mm] 
and vector variables
}
\vskip 1cm
{Asmus K. BISBO~$^\dag$, Hendrik DE BIE~$^\ddag$ and Joris VAN DER JEUGT~$^\dag$}
\end{center}
\vskip 1cm
{$^\dag$~Department of Applied Mathematics, Computer Science and Statistics, Ghent University, Krijgs\-laan 281-S9, B-9000 Gent, Belgium}\\
{$^\ddag$~Department of Electronics and Information Systems, Faculty of Engineering and Architecture,
Ghent University, Krijgslaan 281-S8, B-9000 Gent, Belgium} \\
{\href{mailto:Asmus.Bisbo@UGent.be}{Asmus.Bisbo@UGent.be}, \href{mailto:Joris.VanderJeugt@UGent.be}{Joris.VanderJeugt@UGent.be}, \href{mailto:Hendrik.DeBie@Ugent.be}{Hendrik.DeBie@Ugent.be}} 

\begin{abstract}
The Lie algebra generated by $m$ $p$-dimensional Grassmannian Dirac operators and $m$ $p$-dimensional vector variables is identified as the orthogonal Lie algebra $\mathfrak{so}(2m+1)$.
In this paper, we study the space $\mathcal{P}$ of polynomials in these vector variables, corresponding to an irreducible $\mathfrak{so}(2m+1)$ representation.
In particular, a basis of $\mathcal{P}$ is constructed, using various Young tableaux techniques.
Throughout the paper, we also indicate the relation to the theory of parafermions.
\end{abstract}

\section{Introduction}
\label{sec1}
The theory of Dirac operators and vector variables is robust, in a wide range of fields in mathematics, physics and computer science. Systems with an arbitrary number of such operators are less studied and the complexities quickly mount as the number of operators rise, showing structure not present in the simpler case \cite{Colombo-Sommen-Sabadini-Struppa-2004}.

Traditionally the $p$-dimensional Dirac operator, $D$ and vector variable $X$ are defined as $\displaystyle D=\sum_{i=1}^p \frac{\partial}{\partial x_i} e_i$ and $\displaystyle X=\sum_{i=1}^p x_i e_i$ respectively, where $x_1,\dots, x_p$ are ordinary variables and $e_1,\dots, e_p$, satisfying $e_ie_j+e_je_i=2\delta_{ij}$, are the generators of the complex Clifford algebra. 

The usefulness of Dirac operators and vector variables has led people to search for interesting analogs and modifications. By introducing a reflection group and exchanging the derivatives for Dunkl derivatives, one obtains the so called Dunkl-Dirac operator \cite{DeBie-Oste-VanderJeugt-2018,Orsted-Somberg-Soucek-2009}.

In this paper we consider the $m$ modified Dirac operators $D_i$ and vector variables $\Theta_i$ ($i=1,\ldots,m$) obtained by replacing the ordinary variables with Grassmannian (anticommuting) variables. We shall refer to them as Grassmannian Dirac operators and vector variables.
Systems with one Grassmannian Dirac operator and one Grassmannian vector variable were studied in \cite{Slupinski-1996}.

Dirac operators and vector variables on superspace have also been considered \cite{Coulembier-DeBie-2015}. Here a combination of ordinary and Grassmannian variables are used.

Whereas ordinary Dirac operators and vector variables can be thought of as acting on Clifford algebra valued polynomials, Grassmannian Dirac operators and vector variables act on the space of Clifford algebra valued exterior forms, or equivalently the space of Clifford algebra valued Grassmann polynomials. Such have previously been considered in \cite{Monakhov-2016-1,Monakhov-2016-2,Slupinski-1996}.

Although the setting of $m$ $p$-dimensional Grasmannian Dirac operators and vector variables is quite simple, their study leads to interesting Lie algebraic properties.
First of all, we observe that the set of $2m$ operators $D_i$ and $\Theta_i$ generate the orthogonal Lie algebra $\mathfrak{so}(2m+1)$.
The space $\mathcal{P}$ of polynomials in the (non-commuting) vector variables $\Theta_i$ is also worth studying.
It coincides with the irreducible $\mathfrak{so}(2m+1)$ representation with Dynkin labels $[0,\ldots,0,p]$.
Our main effort in this paper is the construction of an appropriate basis of $\mathcal{P}$, in terms of the vector variables $\Theta_i$.
This construction involves combinatorial techniques, for which several types of Young tableaux are needed.
As a byproduct, we also obtain elegant expressions of the basis elements of $\mathcal{P}$ in terms of the ordinary Grassmann variables $\theta_{i\alpha}$ and the Clifford algebra generators.

The present study can also be framed in the theory of so-called parafermions, introduced in mathematical physics a long time ago~\cite{Green-1953}.
In fact, the Grassmannian vector variables and Dirac operators $\Theta_i$ and $D_i$ are a realization of parafermion creation and annihilation operators $a_i^+$ and $a_i^-$, as they satisfy the same triple relations~\cite{Stoilova-VanderJeugt-2008}.
The basis of $\mathcal{P}$ constructed in this paper, is then a basis of the parafermionic Fock space of order~$p$~\cite{Ohnuki-Kamefuchi-1982,Stoilova-VanderJeugt-2008} purely in terms of parafermionic creation operators acting on a vacuum state.
Although the contents of this paper is mathematical, we shall point out the link to parafermion theory whenever relevant.

\section{Preliminaries}
\label{sec2}

In this section we introduce the fundamental concepts which will be used throughout the paper. 
We begin by introducing Clifford algebras, exterior forms, Grassmann variables, and Grassmannian Dirac operators and vector variables. 
More elaborate expositions on Clifford algebras and exterior forms can be found in \cite{Berezin-1987,Porteous-1982}.
Following this we introduce the combinatorial objects needed: partitions, Young diagrams, Young tableaux and semistandard Young tableaux. Where possible we follow \cite{Macdonald-1995} with respect to the notation and conventions. Finally we introduce the notion of subtableaux of semistandard Young tableaux and use it to define a total ordering of the set of semistandard Young tableaux.

\subsection{Exterior forms and Clifford algebras}
\label{sec2.1}

Throughout this paper $m$ and $p$ will be positive integers. We let $\Lambda[\C^{mp}]$ refer to the space of exterior forms on the vector space $\C^{mp}$, that is $\Lambda[\C^{mp}]$ is the exterior algebra on $\C^{mp}$. The vector space $\Lambda[\C^{mp}]$ has a basis consisting of the $2^{mp}$ elements 
\begin{equation}
\theta^\gamma:=\theta_{11}^{\gamma_{11}}\wedge\cdots \wedge\theta_{m1}^{\gamma_{m1}}\wedge
\theta_{12}^{\gamma_{12}}\wedge\cdots \wedge\theta_{m2}^{\gamma_{m2}}\wedge
\cdots\wedge
\theta_{1p}^{\gamma_{1p}}\wedge\cdots \wedge\theta_{mp}^{\gamma_{mp}},
	\end{equation} 
for $\gamma\in M_{mp}(\Z_2)$. Here $M_{mp}(\Z_2)$ denotes the space of $m$ by $p$ matrices with values in $\Z_2$.
For ease of notation we will suppress the wedge products opting instead to represent such elements as monomials in the Grassmann variables $\theta_{i\alpha}$. In this notation
\begin{equation}
\theta^\gamma:=\theta_{11}^{\gamma_{11}}\cdots \theta_{m1}^{\gamma_{m1}}
\theta_{12}^{\gamma_{12}}\cdots \theta_{m2}^{\gamma_{m2}}\cdots
\theta_{1p}^{\gamma_{1p}}\cdots \theta_{mp}^{\gamma_{mp}},
	\end{equation}
for $\gamma\in M_{mp}(\Z_2)$. Together with the Grassmann variables $\theta_{ij}$ we will consider the corresponding Grassmann derivatives $\partial_{i\alpha}:=\frac{\partial}{\partial \theta_{i\alpha}}$. Grassmann variables and derivatives satisfy the following algebraic relations,
\begin{equation}
\{ \partial_{i\alpha}, \theta_{j\beta}\}=\delta_{ij}\delta_{\alpha\beta},
\quad
\{ \partial_{i\alpha}, \partial_{j\beta}\}=0
\quad \text{ and }\quad
\{ \theta_{i\alpha}, \theta_{j\beta}\}=0,
	\end{equation}
for all $i,j\in\{1,\dots,m\}$ and $\alpha,\beta\in \{1,\dots,p\}$, where $\{a,b\}=ab+ba$ is the anti-commutator bracket. Grassmannian variables and derivatives have a natural action on the space $\Lambda[\C^{mp}]$ given by
\begin{equation}
\theta_{i\alpha}(q):= \theta_{i\alpha}q,\quad 
\partial_{i\alpha}(q'):=\partial_{i\alpha}q' \quad\text{ and }\quad
\partial_{i\alpha}(1):=0,
	\end{equation}
for all $i\in\{1,\dots,m\}$ and $\alpha\in \{1,\dots,p\}$, $q,q'\in\Lambda[\C^{mp}]$ with $q'\neq 0$.

We let $\mathcal{C}\ell_p$ denote the complex Clifford algebra with $p$ generators $e_1,\dots,e_p$ satisfying 
\begin{equation}
\label{sec2_eq_clifford-relation}
\{e_\alpha,e_\beta\}=2\delta_{\alpha\beta},
	\end{equation}
for all $\alpha,\beta\in\{1,\dots,p\}$.
As a vector space $\mathcal{C}\ell_p$ has a basis consisting of the $2^p$ elements 
\begin{equation}
e^\eta:= e_1^{\eta_1}\cdots e_p^{\eta_p},
	\end{equation}
for $\eta\in \Z_2^p$.
The $m$ Grassmannian Dirac operators and $m$ Grassmannian vector variables can now be defined as
\begin{equation}
D_i:=\sum_{\alpha=1}^p \frac{\partial}{\partial\theta_{i\alpha}} e_\alpha
\quad \text{ and } \quad
\Theta_i:=\sum_{\alpha=1}^p \theta_{i\alpha} e_\alpha,
	\end{equation}
for all $i\in\{1,\dots,m\}$. 
These operators act on the space $\mathcal{B}:=\Lambda[\C^{mp}]\otimes \mathcal{C}\ell_p$. Depending on our perspective $\mathcal{B}$ can be called the space of Clifford algebra valued Grassmann polynomials or the space of Clifford algebra valued exterior forms. 
The action of $D_i$ and $\Theta_i$ is defined as follows.
\begin{equation}
D_i(q \otimes f):=\sum_{\alpha=1}^p \partial_{i\alpha}q \otimes e_\alpha f
\quad\text{and}\quad
\Theta_i(q\otimes f):=\sum_{\alpha=1}^p \theta_{i\alpha}q  \otimes e_\alpha f,
	\end{equation}
for all $i\in\{1,\dots,m\}$, $q\in \Lambda[\C^{mp}]$ and $f\in \mathcal{C}\ell_p$. In the future we will be suppressing the tensor product when discussing elements of $\mathcal{B}$.
We make $\mathcal{B}$ into a Hilbert space by endowing it with the Hermitian inner product $\langle\cdot,\cdot\rangle$ defined by 
\begin{equation}
\langle \theta^\gamma e^\eta, \theta^{\gamma'}e^{\eta'}\rangle:= \delta_{\gamma\gamma'}\delta_{\eta\eta'},
	\end{equation}
for all $\gamma,\gamma'\in M_{mp}(\Z_2)$ and $\eta,\eta'\in \Z_2^p$, where the first component is chosen to be antilinear. A short calculation shows that 
\begin{equation}
\langle \Theta_i\theta^\gamma e^\eta, \theta^{\gamma'}e^{\eta'}\rangle
	= \langle \theta^\gamma e^\eta, D_i\theta^{\gamma'}e^{\eta'}\rangle,
	\end{equation}
for all $i\in \{1,\dots,m\}$, $\gamma,\gamma'\in M_{mp}(\Z_2)$ and $\eta,\eta'\in \Z_2^p$. This means that the operators $D_i$ and $\Theta_i$ are each others Hermitian adjoints.

\subsection{Partitions and Young tableaux}
\label{sec2.2}

We let $\mathbb{P}$ denote the set of partitions, that is the set of finite non-increasing sequences on non-negative integers.
\begin{equation}
\mathbb{P}=\big\{\lambda=(\lambda_1,\dots,\lambda_k): k\in\N, \lambda_1\geq\cdots\geq\lambda_k\geq 0\big\}.
	\end{equation}
The length $\ell(\lambda)$ of the partition $\lambda$ is the number of non-zero entries of $\lambda$. We consider two partitions to be equal if all their non-zero entries agree, meaning that they only differ by the number of zeroes at the tail end of the sequences.
A partition $\lambda\in\mathbb{P}$ can be described as a diagram of empty boxes known as a Young diagram, with $\lambda_i$ being the number of boxes in the $i$'th row. 
For example, if $\lambda=(\lambda_1,\lambda_2,\lambda_3)=(4,3,1)$ then the corresponding Young diagram is 
\begin{equation}
\ytableausetup{centertableaux,boxsize=1.1em}
\lambda =
\ydiagram{4,3,1}.
	\end{equation}
To each partition $\lambda\in \mathbb{P}$ we associate the conjugate partition $\lambda'\in \mathbb{P}$, whose $i$'th entry $\lambda_i'$ is defined to be the number of boxes in the $i$'th column of the Young diagram of $\lambda$. So if $\lambda=(4,3,1)$, then $\lambda'=(3,2,2,1)$.
From this perspective $\ell(\lambda)$ and $\ell(\lambda')$ describe the number of rows and columns in the Young diagram $\lambda$ respectively.

A Young tableau of shape $\lambda$ with entries in $\{1,\dots,m\}$ is then a filling of the Young diagram of shape $\lambda$ by numbers from the set $\{1,\dots,m\}$. 
We denote the set of Young tableaux with entries in $\{1,\dots,m\}$ by $\mathbb{E}$.
As an example a Young tableau of shape $(4,3,1)$ and with $m= 4$ we consider
\begin{equation}
\label{sec4_eq_example-Young-tableau}
\ytableausetup{centertableaux,boxsize=1.1em}
\ytableaushort{2431,211,4}\in \mathbb{E}.
	\end{equation}
The weight of a Young tableau $A\in\mathbb{E}$ is then the $m$-tuple $\mu\in\N_0^m$, with $\mu_i$ being the number of times the entry $i$ appears in $A$. The weight of the Young tableau in \eqref{sec4_eq_example-Young-tableau} would then be $(3,2,1,2)$.
For any Young tableau $A\in \mathbb{E}$ we let $\lambda_A$ and $\mu_A$ denote its shape and weight respectively.

We will primarily be interested in the subset $\mathbb{Y}\subset\mathbb{E}$ consisting of semistandard (s.s.) Young tableaux. A Young tableau is called semistandard if its entries are non-decreasing from left to right along each row and increasing from top to bottom along each column. 
We consider for $m=4$ two examples of s.s.\ Young tableaux of shape $(4,3,1)$ and weight $(3,2,1,2)$:
\begin{equation}
\label{sec4_eq_example-ss-Young-tableau}
\ytableausetup{centertableaux,boxsize=1.1em}
\ytableaushort{1112,234,4}\in\mathbb{Y}
\quad\text{ and }\quad
\ytableaushort{1114,223,4}\in\mathbb{Y}.
	\end{equation}

\subsection{Subtableaux and a total ordering of $\mathbb{Y}$}
\label{sec2.3}

Given $A\in\mathbb{Y}$ and $k\in \{1,\dots,m\}$ we define the $k$'th subtableau of $A$ to be the tableau $A^k\in\mathbb{Y}$ obtained by truncating $A$ to only the entries containing the numbers $1,\dots,k$. We take $A^0$ to be the empty tableau.
The following example illustrates subtableaux,
\begin{equation}
A=
\ytableausetup{centertableaux,boxsize=1.1em}
\ytableaushort{1124,234}
\quad
\implies
A^4=A,
\quad
A^3=
\ytableausetup{centertableaux,boxsize=1.1em}
\ytableaushort{112,23}, 
\quad
A^2=
\ytableausetup{centertableaux,boxsize=1.1em}
\ytableaushort{112,2}
\text{ and }
A^1=
\ytableausetup{centertableaux,boxsize=1.1em}
\ytableaushort{11}\ .
	\end{equation}

We endow both $\N_0^m$ and $\mathbb{P}$ with the graded lexicographic ordering, both denoted by $<$. See \Cref{appA} for more details. The total ordering on $\mathbb{Y}$ can now be defined as follows.

Given $A,B\in \mathbb{Y}$ we write $A<B$ if $\mu_A<\mu_B$ in $\N_0^m$, or if $\mu_A=\mu_B$ and there exists $k\in\{1,\dots,m\}$ such that
\begin{equation}
\label{sec2_eq_ordering}
\lambda_{A^{l}}=\lambda_{B^{l}} \text{ and } \lambda_{A^k}<\lambda_{B^k}\text{ in } \mathbb{P},
	\end{equation}
for all $l<k$.
The relation $<$ on $\mathbb{Y}$ defined above is a total order. This follows from the fact that the graded lexicographic order gives a total ordering of both $\N_0^m$ and $\mathbb{P}$.
The following example illustrates how this ordering applies to the 13 s.s. Young tableaux of weight $\mu=(2,1,1,1)$.
\begin{equation*}
\tiny
\ytableausetup{centertableaux,boxsize=1.1em}
\ytableaushort{11,2,3,4}\quad<\quad\ytableaushort{11,24,3}\quad<\quad
\ytableaushort{114,2,3}\quad<\quad\ytableaushort{11,23,4}\quad<\quad
\ytableaushort{114,23}\quad<\quad\ytableaushort{113,2,4}\quad<\quad
\ytableaushort{113,24}\quad<\quad\ytableaushort{1134,2}\quad<\quad\\
	\end{equation*}
\begin{equation*}
\tiny
\ytableausetup{centertableaux,boxsize=1.1em}
<\quad
\ytableaushort{112,3,4}\quad<\quad
\ytableaushort{112,34}\quad<\quad\ytableaushort{1124,3}\quad<\quad
\ytableaushort{1123,4}\quad<\quad\ytableaushort{11234}.
	\end{equation*}

\section{The Lie algebra generated by Grassmannian Dirac operators and vector variables}

In \cite{Slupinski-1996} it was observed that the Lie algebra $\mathfrak{sl}(2)\cong \mathfrak{so}(3)$ can be realized using one Grassmannian Dirac operator and one Grassmannian vector variable. We observe in \Cref{sec3_theo_so(2m+1)} that this result generalizes when one considers $m$ Grassmannian Dirac operators and $m$ Grassmannian vector variables. This realization of has, to the best of the authors' knowledge, not been presented in literature before.

\begin{theo}
\label{sec3_theo_so(2m+1)}
The $2m$ operators $D_i$ and $\Theta_i$, for $i\in\{1,\dots,m\}$, acting on $\mathcal{B}$ satisfy the following commutator relations
\begin{equation}
\label{sec3_eq_so(2m+1)-relations}
\begin{split}
[[D_j,\Theta_k],\Theta_l]&=-2\delta_{jl}\Theta_k,\\
[[D_j,D_k],\Theta_l]&=2\delta_{kl}D_j-2\delta_{jl}D_k, \\
[[D_j,D_k],D_l]&=0, 
	\end{split}
\qquad
\begin{split}
[[D_j,\Theta_k],D_l]&=2\delta_{kl}D_j,\\
[[\Theta_j,\Theta_k],D_l]&=2\delta_{kl}\Theta_j-2\delta_{jl}\Theta_k, \\
[[\Theta_j,\Theta_k],\Theta_l]&=0,
	\end{split}
	\end{equation}
for all $j,k,l\in \{1,\dots,m\}$. They thus generate the Lie algebra $\mathfrak{so}(2m+1)$.
	\end{theo}
\begin{proof}
The Lie algebra $\mathfrak{so}(2m+1)$ can be identified as the algebra with $2m$ generators satisfying the relations \eqref{sec3_eq_so(2m+1)-relations}, see \cite{Kamefuchi-Takahashi-1962,Ohnuki-Kamefuchi-1982,Ryan-Sudarshan-1963} or~\cite{Stoilova-VanderJeugt-2008} for an identification with $\mathfrak{so}(2m+1)$ root vectors. Therefore it remains only to show that $D_i$ and $\Theta_i$, for $i\in\{1,\dots,m\}$, satisfy these relations. Proving this is a matter of simple yet tedious calculations and is thus left to the reader.
	\end{proof}
	
Endowing this realization with the unitary structure given by $D_i^*:=\Theta_i$, for all $i\in\{1,\dots,m\}$, it is clear that $\mathcal{B}$ has the structure of a unitary $\mathfrak{so}(2m+1)$-module.

Note that the relations \eqref{sec3_eq_so(2m+1)-relations} are the so-called triple relations of the parafermionic creation and annihilation operators used in parastatistical field theories, see \cite{Green-1953,Ohnuki-Kamefuchi-1982}. When $p=1$ these relations reduce to those of the ordinary fermionic creation and annihilation operators.

The Lie algebra $\mathfrak{so}(2m+1)$ contains an $m$ dimensional Cartan subalgebra $\mathfrak{h}$ with basis consisting of the elements
\begin{equation}
h_i:= -\frac{1}{2}[D_i,\Theta_i],
	\end{equation}
for all $i\in\{1,\dots,m\}$. We let $\epsilon_i\in \mathfrak{h}^*$, for $i\in\{1,\dots,m\}$, be the corresponding dual basis. 
Given a weight $\mu\in \mathfrak{h}^*$ we can write $\mu=\sum_{i=1}^m \mu_i\epsilon_i$, or more simply $(\mu_1,\dots,\mu_m)$.
The root system of $\mathfrak{so}(2m+1)$ is
\begin{equation}
\big\{\epsilon_i\pm\epsilon_j, \pm \epsilon_k: i,j,k\in\{1,\dots,m\}, i\neq j\big\},
	\end{equation}
with simple roots $\{\epsilon_1-\epsilon_2,\dots,\epsilon_{m-1}-\epsilon_m,\epsilon_m\}$. The positive root vectors are then~\cite{Stoilova-VanderJeugt-2008}  
\begin{equation}
\label{sec3_eq_positive-root-vectors}
\big\{[\Theta_i,\Theta_j],[\Theta_i,D_j],\Theta_k: i,j,k\in\{1,\dots,m\}, i< j\big\}.
	\end{equation}

\section{A simple module of polynomials in the $\Theta_i$'s}
\label{sec4}

Due to \Cref{sec3_theo_so(2m+1)} we know that the operators $\Theta_i$ and $D_i$, for $i\in\{1,\dots,m\}$, generate a copy of the Lie algebra $\mathfrak{so}(2m+1)$. This allows us to decompose the module $\mathcal{B}$ into a direct sum of simple $\mathfrak{so}(2m+1)$-modules. 
For the remainder of this paper we study a component in this decomposition, which carries information about the Grassmannian Dirac operators and vector variables. We consider in particular the subspace of $\mathcal{B}$ consisting of polynomials in the Grassmannian vector variables $\Theta_1,\dots,\Theta_m$:
\begin{equation}
	\mathcal{P}:= \Span\big\{ \Theta_{i_1}\cdots \Theta_{i_k}(1): k\in\N_0, i_1,\dots,i_k\in\{1,\dots,m\}\big\}.
	\label{41}
	\end{equation}
We show that this is a simple module of $\mathfrak{so}(2m+1)$, see \Cref{sec4_prop_P-simple-module}, and present the character formula and weight space dimensions of $\mathcal{P}$, see \eqref{sec4_eq_character-formula-1} and \eqref{sec4_eq_weight-space-dimension}. In \Cref{sec5} we construct a basis for $\mathcal{P}$ consisting of polynomials in the Grassmannian vector variables. The corresponding problem for the usual Euclidean vector variables was solved in \cite{Bisbo-DeBie-VanderJeugt-2021}.
As a vector space $\mathcal{P}$ is isomorphic to the unital associative algebra $\mathcal{A}(\Theta)$ generated by the Grassmannian vector variables $\Theta_1,\dots,\Theta_m$ considered as operators acting on $\mathcal{B}$. 
The isomorphism is given concretely by letting the operators in $\mathcal{A}(\Theta)$ act on the constant polynomial $1\in \mathcal{P}$.
The Hermitian conjugate gives an anti-isomorphism between $\mathcal{A}(\Theta)$ and the unital associative algebra $\mathcal{A}(D)$ generated by the Grassmannian Dirac operators $D_1,\dots,D_m$. 
The point of these observations is that by constructing a basis for the $\mathfrak{so}(2m+1)$-module we automatically obtain concrete bases, in the vector space sense, for the algebras $\mathcal{A}(\Theta)$ and $\mathcal{A}(D)$. See \cite{Sommen-Acker-1992,Sommen-1997} for work dealing with such algebras and the operators therein for the case of the usual Euclidean Dirac operators and vector variables.
The space $\mathcal{P}$ is furthermore of interest in the study of parastatistical field theories. 

\begin{prop}
\label{sec4_prop_P-simple-module}
The space $\mathcal{P}$ is a simple unitary lowest weight module of $\mathfrak{so}(2m+1)$ with lowest weight vector $1\in \mathcal{P}$ of weight $(-\frac{p}{2},\dots,-\frac{p}{2})$.
	\end{prop}
\begin{proof}
Recall that $\partial_{i\alpha}(1)=0$, for all $i\in\{1,\dots,m\}$ and $\alpha\in\{1,\dots,p\}$.  Using this a short calculation shows that $D_i(1)=0$ and $h_i(1)=-\frac{p}{2}$, for all $i\in\{1,\dots,m\}$. The relations \eqref{sec3_eq_so(2m+1)-relations} imply that $\mathcal{P}$ is invariant under the action of the simple Lie algebra $\mathfrak{so}(2m+1)$. Finally, the simplicity of $\mathcal{P}$ follows since it is a submodule, generated from one element, of the module $\mathcal{B}$ which is finite and thus semisimple.
	\end{proof}
In \cite{Stoilova-VanderJeugt-2008} the module $\mathcal{P}$ was studied and the following character formula was obtained:
\begin{equation}
\label{sec4_eq_character-formula-1}
\Char \mathcal{P}
	= (t_1\cdots t_m)^{-p/2}\sum_{\substack{\lambda\in\mathbb{P}, \\ \ell(\lambda') \leq p}} s_\lambda(t_1,\dots,t_m),
	\end{equation}
where $s_\lambda$ denotes the Schur function indexed by the partition $\lambda$ and $t_i$ denotes the formal exponential $e^{\epsilon_i}$.
Applying the monomial expansion of $s_\lambda$ to \eqref{sec4_eq_character-formula-1}, see \cite{Macdonald-1995}, we get
\begin{equation}
\label{sec4_eq_character-formula-2}
\Char \mathcal{P}
	= (t_1\cdots t_m)^{-p/2}\sum_{\substack{\lambda\in\mathbb{P}, \\ \ell(\lambda') \leq p}} 
	\sum_{\mu\in \N_0^m} K_{\lambda\mu}t_1^{\mu_1}\cdots t_m^{\mu_m},
	\end{equation}
where $K_{\lambda\mu}$ denotes the Kostka number:
\begin{equation}
K_{\lambda\mu}:=\#\big\{\text{ s.s.\ Young tableaux in } \mathbb{Y} \text{ of shape } \lambda \text{ and weight $\mu$ } \big\}.
	\end{equation}
The character formula \eqref{sec4_eq_character-formula-2} makes it clear that the set of weights of $\mathcal{P}$ is given by
\begin{equation}
\big\{ \mu-p/2:=(\mu_1-p/2,\dots,\mu_m-p/2): \mu\in\{0,\dots,p\}^m\big\}
	\end{equation}
and that the corresponding weight space dimensions are 
\begin{equation}
\label{sec4_eq_weight-space-dimension}
\begin{split}
\dim \mathcal{P}_{\mu-\frac{p}{2}} 
	&= \sum_{\substack{\lambda\in\mathcal{P}, \\ \ell(\lambda') \leq p}}  K_{\lambda\mu}\\
	&= \#\big\{ \text{ s.s. Young tableaux of weight $\mu$ and with at most $p$ columns }\big\}.  
	\end{split}
	\end{equation}

\section{A basis for the module $\mathcal{P}$}
\label{sec5}

We now turn to the construction of an appropriate basis for the module $\mathcal{P}$ consisting of vectors $\omega_A$ indexed by s.s.\ Young tableaux.
Of course, other bases for $\mathfrak{so}(2m+1)$-modules have been considered in the literature. We briefly present an overview of known results.

In \cite{Stoilova-VanderJeugt-2008} a Gel'fand-Zetlin basis for $\mathcal{P}$, parameterized by Gel'fand-Zetlin patterns, was considered and matrix elements of the $\mathfrak{so}(2m+1)$-action on this basis were calculated. 
In a more general context, monomial bases have been constructed for $\mathfrak{so}(2m+1)$-modules. 
The crystal bases coming from the canonical basis of the quantum group $U_q(\mathfrak{g})$ presents a well known method for constructing monomial bases for finite dimensional modules of a semisimple Lie algebra $\mathfrak{g}$, see \cite{Lusztig-1990-1,Lusztig-1990-2}. Another method was developed and applied to simple modules of Lie algebras of type A and C in the papers \cite{Feigin-Fourier-Littlemann-2011-1,Feigin-Fourier-Littlemann-2011-2,Feigin-Fourier-Littlemann-2017}. Application to cases B and D can be found in \cite{Gornitskii-2019,Makhlin-2019}.

The vectors of a monomial basis are defined as monomials in the negative root vectors of $\mathfrak{so}(2m+1)$ acting on the highest weight vector of the relevant module, or equivalently as monomials in the positive root vectors, \eqref{sec3_eq_positive-root-vectors}, of $\mathfrak{so}(2m+1)$ acting on the lowest weight vector.

The basis that we wish to construct here for $\mathcal{P}$ differs from all known bases.  It differs specifically from the monomial bases in that every basis element is defined as a polynomial in only the positive generators, $\Theta_1,\dots,\Theta_m$, of $\mathfrak{so}(2m+1)$ acting on the lowest weight vector. This means that the remaining positive root vectors $[\Theta_i,\Theta_j]$ and $[D_i,\Theta_j]$, for $i<j$, are not needed for the definition of the basis. 
In the present context, with $\mathcal{P}$ given by~\eqref{41}, this is natural to consider.

Such a basis for $\mathcal{P}$ is also important in the context of parafermions: the basis vectors expressed as polynomials in the $\Theta_i$ acting on $1$ translate to polynomials in the parafermionic creation operators $a_i^+$ acting on a vacuum state, and thus form an a new and interesting basis of the parafermionic Fock space of order~$p$~\cite{Ohnuki-Kamefuchi-1982,Stoilova-VanderJeugt-2008}.

In \Cref{sec5.1} we define the vector $\omega_A$, given a Young tableau $A$. 
Following that we prove in \Cref{sec5.2} that such vectors with $A\in\mathbb{Y}$ and $\ell(\lambda_A')\leq p$ form a basis for $\mathcal{P}$. 
The main difficulty of the proof lies in the identification of a certain leading monomial of $\omega_A$ and proving that it `respects' the total order on $\mathbb{Y}$. This is the content of \Cref{sec5_prop_leading-term}, the proof of which we postpone until \Cref{sec5.4}. In \Cref{sec5.3} we define row distinct and $A$-restricted Young tableaux, and use these notions to obtain monomial expansions of $\omega_A$ necessary for proving \Cref{sec5_prop_leading-term}.

\subsection{Construction of the vectors $\omega_A$}
\label{sec5.1}

We are interested in constructing the vectors $\omega_A$, for $A\in\mathbb{E}$, such that each entry $i$ in $A$ correspond to an occurrence of the Grassmannian vector variable $\Theta_i$. 
To do so, we let $A(k,l)$ denote the entry of $A$ in the $k$'th row and $l$'th column. If the shape of $A$ is $\lambda$, then $(k,l)$ runs over the coordinates of the boxes of the Young diagram of $\lambda$. In a slight abuse of notation we shall also use $\lambda$ to denote the set of these coordinates:
\begin{equation}
\lambda=\big\{ (k,l) : k\in\{1,\dots,\ell(\lambda)\}, l\in\{1,\dots, \lambda_k\}\big\}.
	\end{equation}

For each $\lambda\in \mathbb{P}$ we consider the permutation group
\begin{equation}
S_\lambda=S_{\lambda_1}\times\cdots S_{\lambda_{\ell(\lambda)}}.
	\end{equation}
The purpose of this group is to permute the rows of any given Young tableau $A\in\mathbb{E}$ whose shape is $\lambda$.
The group $S_\lambda$ is called the Young subgroup associated to the partition $\lambda$. 
Such permutation groups are widely used in the classification of irreducible representations of the symmetric group, see \cite{James-1978}. 
The group $S_\lambda$ acts on the set of Young tableaux of shape $\lambda$ in the following manner.	
Given a Young tableau $A\in\mathbb{E}$ of shape $\lambda$ and a permutation $\tau=(\tau_1,\dots,\tau_{\ell(\lambda)})\in S_\lambda$ we define the row permuted Young tableau $A^\tau\in\mathbb{E}$ to be the tableau with entries 
\begin{equation}
A^\tau(k,l):=A(k,\tau_k(l)),
	\end{equation}
for all $(k,l)\in\lambda$. 
We can now define the vector $\omega_A$ as follows. For any $A\in\mathbb{E}$ of shape $\lambda\in\mathbb{P}$ let 
\begin{equation}
\Theta_A:=
	\big(\Theta_{A(1,1)}\cdots\Theta_{A(1,\lambda_1)}\big)
	\big(\Theta_{A(2,1)}\cdots\Theta_{A(2,\lambda_2)} \big)
	\cdots
	\big(\Theta_{A(\ell(\lambda),1)}\cdots\Theta_{A(\ell(\lambda),\lambda_{\ell(\lambda)})} \big)
	\in \mathcal{A}(\Theta)
	\end{equation}
and
\begin{equation}
\label{sec5_eq_defi-omegaA}
\omega_A:=
	\frac{1}{A!}
	\sum_{\tau\in S_\lambda} \Theta_{A^\tau}(1)\in \mathcal{P},
	\end{equation}
where $A!$ is the following factorial 
\begin{equation}
A!:=\lambda_1!\cdots\lambda_{\ell(\lambda)}!\prod_{i=1}^m\prod_{k=1}^{\ell(\lambda)} 
	\big( (\lambda_{A^{i}})_k-(\lambda_{A^{i-1}})_k \big)!.
	\end{equation}
Here it is interesting to note that
\begin{equation}
\label{sec4_eq_i's-in-the-k'th-row}
(\lambda_{A^{i}})_k-(\lambda_{A^{i-1}})_k=\#\{\text{ $i$'s in the $k$'th row of $A$ }\}.
	\end{equation}
	
We remark at this point that $\omega_A\neq 0$ if and only if $\ell(\lambda')\leq p$. This will appear later as a direct consequence of \Cref{sec5_lemm_monomial-expansion-tableaux-1}. Keeping \eqref{sec4_eq_weight-space-dimension} in mind, this means that we have defined the right number of non-zero vectors $\omega_A$, for $A\in\mathbb{Y}$ and $\ell(\lambda_A')\leq p$, to form a basis for $\mathcal{P}$. Of course, it still remains to prove linear independence, which will take up the rest of this section.

As an example we calculate $\omega_A$, where
\begin{equation}
\label{sec5_eq_example-A}
\ytableausetup{centertableaux,boxsize=1.1em}
	A=\ytableaushort{113,22},
	\end{equation}
in which case
\begin{equation}
\label{sec5_eq_example-omegaA}
\begin{split}
\omega_A
	&= 
	\frac{1}{3!2!1!2!2!}
	\left(
	2!2!\Theta_{\tiny\ytableaushort{113,22}}
	+2!2!\Theta_{\tiny\ytableaushort{131,22}}
	+2!2!\Theta_{\tiny\ytableaushort{311,22}}
	\right)
	\\&=
	\frac{1}{12}
	\left(
	\Theta_1^2\Theta_3\Theta_2^2 + \Theta_1\Theta_3\Theta_1\Theta_2^2 +\Theta_3\Theta_1^2\Theta_2^2
	\right).
	\end{split}
	\end{equation}	
Expanding $\omega_A$ into terms of the form $\theta^\gamma e^\eta$ is at this moment very tedious and computationally heavy.
In \Cref{sec5.3} we will see that \Cref{sec5_lemm_monomial-expansion-tableaux-2} makes such calculations much more manageable, as can be seen in the example in equation \eqref{sec5_eq_example-expansion-A-restricted}.

\subsection{The leading monomial of $\omega_A$ and proof of basis}
\label{sec5.2}

For the remainder of this section we will focus our attention on vectors $\omega_A$ corresponding to s.s. Young tableaux $A\in\mathbb{Y}$ with $\ell(\lambda_A')\leq p$.
Recall from \Cref{sec2} that the monomials $\theta^\gamma e^\eta$, for $\gamma\in M_{mp}(\Z_2)$ and $\eta\in\Z_2^p$, form a basis for the module $\mathcal{B}$. This allows us to make the following expansion
\begin{equation}
\label{sec5_eq_monomial-expansion}
\omega_A= 
	\sum_{\gamma,\eta} \langle \theta^\gamma e^\eta, \omega_A\rangle \theta^\gamma e^\eta,
	\end{equation}
for all $A\in\mathbb{Y}$ with $\ell(\lambda_A')\leq p$.

Each s.s. Young tableau $A\in\mathbb{Y}$ with $\ell(\lambda_A')\leq p$ can be identified with the matrix $\gamma_A\in M_{mp}(\Z_2)$ whose entries are given as follows
\begin{equation}
\label{sec5_eq_leading-term-exponent-matrix}
(\gamma_A)_{ij}:=\#\{ \text{ $i$'s in the $j$'th column of $A$ }\},
	\end{equation}
for all $i\in\{1,\dots,m\}$ and $j\in\{1,\dots,p\}$. Using this we define the {\em leading monomial} of $\omega_A$ to be $\theta^{\gamma_A} e^{\lambda_A'}$. 
As an exponent of $e^{\lambda_A'}$, $\lambda_A'$ is considered modulo $2$, that is as an element of $\Z_2^p$. This is possible due to the relation \eqref{sec2_eq_clifford-relation}. As an example, consider the tableau $A$ from \eqref{sec5_eq_example-A} and assume that $p=3$. Then $\lambda_A'=(2,2,1)$ and $e^{\lambda_A'}=e_1^2e_2^2e_3=e_3$.

The following result tells us that the leading monomials `respect' the ordering of $\mathbb{Y}$ defined in \Cref{sec2.3} and that they always appear with coefficient $1$ in the expansion \eqref{sec5_eq_monomial-expansion} of their respective vectors.
To state the result we let $\eta_\gamma$, for $\gamma\in M_{mp}(\Z_2)$, denote the column sum of $\gamma$, that is
\begin{equation}
\eta_\gamma = 
	\left(
	\sum_{i=1}^m \gamma_{i1},\dots,\sum_{i=1}^m \gamma_{ip}
	\right).
	\end{equation}

\begin{prop}
\label{sec5_prop_leading-term}
Suppose $A\in \mathbb{Y}$ with $\ell(\lambda_A')\leq p$, then
\begin{equation}
\label{sec5_eq_leading-term}
\omega_A = 
	\sum_{\gamma\in M_{mp}(\Z_2)} \langle \theta^\gamma e^{\eta_\gamma}, \omega_A \rangle\theta^\gamma e^{\eta_\gamma}.
	\end{equation}
Additionally, if $B\in \mathbb{Y}$ with $\ell(\lambda_B')\leq p$ and $A<B$, then 
\begin{equation}
\big\langle \theta^{\gamma_A} e^{\lambda_A'}, \omega_A \big\rangle = 1
\quad\text{ and }\quad 
\big\langle \theta^{\gamma_A} e^{\lambda_A'}, \omega_B \big\rangle = 0.
	\end{equation}
	\end{prop}
\begin{proof}
Proving this result is a rather technical endeavor, which relies on combinatorial properties of row distinct and $A$-restricted Young tableaux, and their relation to the vectors $\omega_A$. 
The proof of this proposition is postponed until \Cref{sec5.4}, before which row distinct and $A$-restricted Young tableaux will be defined and treated in detail in \Cref{sec5.3}.
	\end{proof}

As a consequence of Proposition \ref{sec5_prop_leading-term} we get linear independence of the $\omega_A$.
\begin{coro}
\label{sec5_coro_linear-independence}
The vectors $\omega_A$, for $A\in\mathbb{Y}$ with $\ell(\lambda_A')\leq p$, are linearly independent.
	\end{coro}

\begin{theo}
\label{sec5_theo_basis}
The $\mathfrak{so}(2m+1)$-module $\mathcal{P}$ has a basis
\begin{equation}
\big\{ \omega_A : A\in \mathbb{Y}, \ell(\lambda_A')\leq p \big\},
	\end{equation}
consisting of weight vectors. The weight of $\omega_A$ is $\mu_A-\frac{p}{2}\in \N_0^m$.
	\end{theo}
\begin{proof}
Using the relations \eqref{sec3_eq_so(2m+1)-relations} it is clear that $h_i\omega_A= ((\mu_A)_i-\frac{p}{2})\omega_A$, for all $i\in\{1,\dots,m\}$. Thus the weight of $\omega_A$ is $\mu_A-\frac{p}{2}$. Together \eqref{sec4_eq_weight-space-dimension} and \Cref{sec5_coro_linear-independence} then imply that the weight space $\mathcal{P}_{\mu-\frac{p}{2}}$ has the basis 
\begin{equation}
\big\{ \omega_A : A\in \mathbb{Y}, \ell(\lambda_A')\leq p, \mu_A=\mu \big\},
	\end{equation}
for all $\mu\in\N_0^m$. Taking the union of the bases for each weight space then yields the desired basis for $\mathcal{P}$.
	\end{proof}
In \Cref{sec4} we mentioned that the basis for the module $\mathcal{P}$ obtained in \Cref{sec5_theo_basis} automatically translates to bases, in the vector space sense, of the unital associative algebras $\mathcal{A}(\Theta)$ and $\mathcal{A}(D)$ generated respectively by the operators $\Theta_1,\dots,\Theta_m$ and $D_1,\dots,D_m$ acting on $\mathcal{B}$. 
For any $A\in\mathbb{E}$ let 
\begin{equation}
	D_A:=
	\big(D_{A(\ell(\lambda),\lambda_{\ell(\lambda)})}\cdots D_{A(\ell(\lambda),1)} \big)
	\cdots
	\big(D_{A(2,\lambda_2)}\cdots D_{A(2,1)} \big)
	\big(D_{A(1,\lambda_1)}\cdots D_{A(1,1)} \big)
	\end{equation}
and note that $(\Theta_A)^*=D_A$, for all $A\in\mathbb{E}$. We then get the following corollary to \Cref{sec5_theo_basis}.
\begin{coro}
As vector spaces $\mathcal{A}(\Theta)$ and $\mathcal{A}(D)$ have bases
\begin{equation}
\left\{
	\sum_{\tau\in S_{\lambda_A}} \Theta_{A^\tau} : A\in \mathbb{Y}, \ell(\lambda_A')\leq p
	\right\}
\quad\text{and}\quad
	\left\{
	\sum_{\tau\in S_{\lambda_A}} D_{A^\tau} : A\in \mathbb{Y}, \ell(\lambda_A')\leq p
	\right\}
\end{equation}
respectively.
	\end{coro}

\subsection{Row distinct and $A$-restricted Young tableaux}
\label{sec5.3}

In order to prove \Cref{sec5_prop_leading-term} we need to know more about how $\omega_A$ expands into a linear combination of monomial terms $\theta^\gamma e^\eta$. 
To do so we introduce the notions of row distinct and $A$-restricted Young tableaux which nicely index the contributions relevant to calculate the coefficients in \eqref{sec5_eq_leading-term}.

We call a Young tableau $A$ row distinct (r.d.) if all entries of each row are distinct, that is, if $A(k,l)\neq A(k,l')$ for $(k,l),(k,l')\in\lambda_A$ with $l\neq l'$.
We denote the set of row distinct Young tableaux with entries in $\{1,\dots,p\}$ by $\mathbb{T}$. 
As an example consider for $p=4$ the following two examples of r.d.\ Young tableaux of shape $(4,3,2)$.
\begin{equation}
\ytableausetup{centertableaux,boxsize=1.1em}
	\ytableaushort{1432,123,31}\in\mathbb{T}
	\quad\text{ and }\quad 
	\ytableaushort{1423,214,34}\in\mathbb{T}.
	\end{equation}
We emphasize here that whereas the tableaux in $\mathbb{Y}$ and $\mathbb{E}$ have entries in $\{1,\dots,m\}$, the tableaux in $\mathbb{T}$ have entries in $\{1,\dots,p\}$.

Given a s.s.\ Young tableau $A\in \mathbb{Y}$ and a r.d.\ Young tableau $C\in\mathbb{T}$, both of shape $\lambda$, we define the following monomials in $\Lambda[\C^{mp}]$ and $\mathcal{C}\ell_p$. Let 
\begin{equation}
\label{sec5_eq_grassmann-monomial-tableaux}
\theta_{AC}:= 
	\prod_{k=1,\dots,\ell(\lambda)}^\rightarrow
	\left(\theta_{A(k,1),C(k,1)}\cdots \theta_{A(k,\lambda_k),C(k,\lambda_k)}\right)
	\in \Lambda[\C^{mp}]
	\end{equation}
and
\begin{equation}
\label{sec5_eq_clifford-monomial-tableaux}
e_C:= 
	\prod_{k=1,\dots,\ell(\lambda)}^\rightarrow
	\left( e_{C(k,1)}\cdots e_{C(k,\lambda_k)}\right)
	\in \mathcal{C}\ell_p,
	\end{equation}
where the arrow refers to the order in which the terms are multiplied ($k = 1$ is leftmost and $k = \ell(\lambda)$ is rightmost).
As examples of such monomials consider the case $m=4$, $p=4$,
\begin{equation}
\label{sec5_eq_r.d.-tableau-example-1}
\ytableausetup{centertableaux,boxsize=1.1em}
A=\ytableaushort{122,234}\in\mathbb{Y}
	\quad\text{ and }\quad
	C=\ytableaushort{341,214}\in\mathbb{T}.
	\end{equation}

The monomials, defined in \eqref{sec5_eq_grassmann-monomial-tableaux} and \eqref{sec5_eq_clifford-monomial-tableaux}, 
then take the form
\begin{equation}
\theta_{AC}=\theta_{13}\theta_{24}\theta_{21}\theta_{22}\theta_{31}\theta_{44}
	=-\theta_{21}\theta_{31}\theta_{22}\theta_{13}\theta_{24}\theta_{44}
	\end{equation}
and
\begin{equation}
e_C=e_3e_4e_1e_2e_1e_4=-e_2e_3.
	\end{equation}
These definitions lead us to the following explicit expansion of $\omega_A$ in terms of monomials $\theta_{AC}e_C$.
\begin{lemm}
\label{sec5_lemm_monomial-expansion-tableaux-1}
For all $A\in \mathbb{Y}$ of shape $\lambda$ with $\ell(\lambda')\leq p$, we have
\begin{equation}
\label{sec5_eq_monomial-expansion-tableaux-1}
\omega_A=\frac{\lambda_1!\cdots \lambda_{\ell(\lambda)}!}{A!}\sum_{C\in\mathbb{T},\ \lambda_C=\lambda} \theta_{AC}e_C.
	\end{equation}
	\end{lemm}
\begin{proof}
We make the following calculation based on \eqref{sec5_eq_defi-omegaA}.
\begin{align*}
\omega_A
	&=
	\frac{1}{A!}
	\prod_{k=1,\dots,\ell(\lambda)}^\rightarrow
	\left(
	\sum_{\sigma\in S_{\lambda_k}} \Theta_{A(k,\sigma(1))}\cdots\Theta_{A(k,\sigma(\lambda_k))}
	\right)
	\\&=
	\frac{\lambda_1!\cdots \lambda_{\ell(\lambda)}!}{A!}
	\prod_{k=1,\dots,\ell(\lambda)}^\rightarrow
	\left(
	\sum^{}{}^{'}\
	\theta_{A(k,1),\alpha_1}\cdots\theta_{A(k,\lambda_k),\alpha_{\lambda_k}}e_{\alpha_1}\cdots e_{\alpha_{\lambda_k}}
	\right)
	\\&=
	\frac{\lambda_1!\cdots \lambda_{\ell(\lambda)}!}{A!}
	\sum_{C\in\mathbb{T},\ \lambda_C=\lambda}
	\left(
	\prod_{k=1,\dots,\ell(\lambda)}^\rightarrow
	\left(\theta_{A(k,1),C(k,1)}\cdots \theta_{A(k,\lambda_k),C(k,\lambda_k)}e_{C(k,1)}\cdots e_{C(k,\lambda_k)}\right)
	\right)
	\\&=
	\frac{\lambda_1!\cdots \lambda_{\ell(\lambda)}!}{A!}\sum_{C\in\mathbb{T},\ \lambda_C=\lambda} \theta_{AC}e_C,
	\end{align*}
where $\sum'$ means that we are summing over elements $\alpha_1,\dots,\alpha_k\in\{1,\dots,p\}$ for which $\alpha_i\neq\alpha_j$ when $i\neq j$.
	\end{proof}

Continuing the example started before \Cref{sec5_lemm_monomial-expansion-tableaux-1}, we illustrate that some of the terms in the expansion \eqref{sec5_eq_monomial-expansion-tableaux-1} equal zero and some are identical. Specifically, if
\begin{equation}
\label{sec5_eq_r.d.-tableau-example-2}
\ytableausetup{centertableaux,boxsize=1.1em}
	C=\ytableaushort{341,214}\in\mathbb{T},
	\quad
	C'=\ytableaushort{342,214}\in\mathbb{T}
	\quad\text{ and }\quad
	C''=\ytableaushort{314,214}\in\mathbb{T},
	\end{equation}
then 
\begin{equation}
\theta_{AC'}
	=
	\theta_{13}\theta_{24}\theta_{22}\theta_{22}\theta_{31}\theta_{44}
	=0,
	\end{equation}
since $\theta_{22}^2=0$, and
\begin{equation}
\theta_{AC''}e_{C''}=\theta_{AC}e_C.
	\end{equation}

To get an expansion of $\omega_A$ that takes these observations into account we define, for $A\in\mathbb{Y}$, the notion of $A$-restricted Young tableaux.

Given a r.d.\ Young tableau $C\in\mathbb{T}$ of shape $\lambda_A$, we say that $C$ is {\em $A$-restricted} if any two entries in $C$ are distinct whenever the corresponding entries in $A$ are equal, and if any two entries on the same row of $C$ are distinct and increasing from left to right whenever the corresponding entries in $A$ are equal.
These conditions can be written rigorously in the following manner.
\begin{equation}
C(k,l)\neq C(k',l'),
	\end{equation}
for all $(k,l),(k',l')\in \lambda_A$ with $(k,l)\neq (k',l')$ and $A(k,l)=A(k',l')$; and 
\begin{equation}
C(k,l) < C(k,l'),
	\end{equation}
for all $(k,l),(k,l')\in \lambda_A$ with $l<l'$ and $A(k,l)=A(k,l')$.
We denote the set of $A$-restricted Young tableaux by $\mathbb{T}_A$. 
Note as an example that of the tableaux $C$, $C'$ and $C''$, defined in \eqref{sec5_eq_r.d.-tableau-example-1} and \eqref{sec5_eq_r.d.-tableau-example-2}, only $C''$ is $A$-restricted with respect to the $A$ defined in \eqref{sec5_eq_r.d.-tableau-example-1}.
\begin{lemm}
\label{sec5_lemm_monomial-expansion-tableaux-2}
For all $A\in \mathbb{Y}$ with $\ell(\lambda_A')\leq p$, we have
\begin{equation}
\label{sec5_eq_monomial-expansion-tableaux-2}
\omega_A=\sum_{C\in\mathbb{T}_A} \theta_{AC}e_C.
	\end{equation}
In addition $\theta_{AC}e_C\neq 0$, for all $C\in\mathbb{T}_A$.
	\end{lemm}
\begin{proof}
To get this expansion we start from \eqref{sec5_eq_monomial-expansion-tableaux-1}.
Given $C\in \mathbb{T}$ with $\lambda_A=\lambda_C$ it is clear that $\theta_{AC}e_C=0$ if and only if there exists $(k,l),(k',l')\in \lambda_A$ such that $(k,l)\neq (k',l')$, $A(k,l)=A(k',l')$ and $C(k,l)=C(k',l')$.

To get the identity \eqref{sec5_eq_monomial-expansion-tableaux-2} we recall from the proof of \Cref{sec5_lemm_monomial-expansion-tableaux-2} that 
\begin{align*}
\omega_A
	&=
	\frac{\lambda_1!\cdots \lambda_{\ell(\lambda)}!}{A!}
	\prod_{k=1,\dots,\ell(\lambda)}^\rightarrow
	\left(
	\sum^{}{}^{'}
	\theta_{A(k,1),\alpha_1}\cdots\theta_{A(k,\lambda_k),\alpha_{\lambda_k}}e_{\alpha_1}\cdots e_{\alpha_{\lambda_k}}
	\right).
	\end{align*}
Noting furthermore that 
\begin{equation}
\theta_{i,\alpha_1}\cdots\theta_{i,\alpha_j}e_{\alpha_1}\cdots e_{\alpha_j}
	=
	\theta_{i,\alpha_{\sigma(1)}}\cdots\theta_{i,\alpha_{\sigma(j)}}e_{\alpha_{\sigma(1)}}\cdots e_{\alpha_{\sigma(j)}},
	\end{equation}
for all $i\in\{1,\dots,m\}$, $\alpha_1,\dots,\alpha_j\in\{1,\dots,p\}$ and $\sigma\in S_j$, and recalling \eqref{sec4_eq_i's-in-the-k'th-row} we can write
\begin{align*}
\omega_A
	&=
	\prod_{k=1,\dots,\ell(\lambda)}^\rightarrow
	\left(
	\sum^{}{}^{''}\
	\theta_{A(k,1),\alpha_1}\cdots\theta_{A(k,\lambda_k),\alpha_{\lambda_k}}e_{\alpha_1}\cdots e_{\alpha_{\lambda_k}}
	\right)
	\\&=
	\sum_{C\in\mathbb{T}_A}
	\left(
	\prod_{k=1,\dots,\ell(\lambda)}^\rightarrow
	\left(\theta_{A(k,1),C(k,1)}\cdots \theta_{A(k,\lambda_k),C(k,\lambda_k)}e_{C(k,1)}\cdots e_{C(k,\lambda_k)}\right)
	\right)
	\\&=	
	\sum_{C\in\mathbb{T}_A} \theta_{AC}e_C,
	\end{align*}
where $\sum''$ means that we are summing over elements $\alpha_1,\dots,\alpha_k\in\{1,\dots,p\}$ for which $\alpha_i\neq\alpha_j$ when $i\neq j$, and $\alpha_i<\alpha_j$ when $i<j$ and $A(k,i)=A(k,j)$.
	\end{proof}
	
To illustrate the use of \Cref{sec5_lemm_monomial-expansion-tableaux-2} we continue the example of equation \eqref{sec5_eq_example-omegaA} in the case $p=3$. Here
\begin{equation}
\ytableausetup{centertableaux,boxsize=1.1em}
	A=\ytableaushort{113,22}.
	\end{equation}
We first note that $\mathbb{T}_A$ contains $9$ elements:
\begin{equation}
\ytableausetup{centertableaux,boxsize=1.1em}
\mathbb{T}_A=
	\left\{
	\tiny\ytableaushort{123,12},\
	\tiny\ytableaushort{123,13},\
	\tiny\ytableaushort{123,23},\
	\tiny\ytableaushort{132,12},\
	\tiny\ytableaushort{132,13},\
	\tiny\ytableaushort{132,23},\
	\tiny\ytableaushort{231,12},\
	\tiny\ytableaushort{231,13},\
	\tiny\ytableaushort{231,23}
	\right\}.
	\end{equation}
Using \Cref{sec5_lemm_monomial-expansion-tableaux-2} we can then calculate the expansion of $\omega_A$ into monomials $\theta^\gamma e^\eta$.
\begin{equation}
\label{sec5_eq_example-expansion-A-restricted}
\begin{split}
\ytableausetup{centertableaux,boxsize=0.9em}
\omega_A&=\theta_{\left(\text{ }
\tiny
\ytableaushort{113,22}
\text{ , }
\ytableaushort{123,12}\text{ }\right)}
e_{\tiny
\ytableaushort{123,12}}
	+\theta_{\left(\text{ }
\tiny
\ytableaushort{113,22}
\text{ , }
\ytableaushort{123,13}\text{ }\right)}
e_{\tiny
\ytableaushort{123,13}}
	+\theta_{\left(\text{ }
\tiny
\ytableaushort{113,22}
\text{ , }
\ytableaushort{123,23}\text{ }\right)}
e_{\tiny
\ytableaushort{123,23}}\\\\
	&\quad+\theta_{\left(\text{ }
\tiny
\ytableaushort{113,22}
\text{ , }
\ytableaushort{132,12}\text{ }\right)}
e_{\tiny
\ytableaushort{132,12}}
	+\theta_{\left(\text{ }
\tiny
\ytableaushort{113,22}
\text{ , }
\ytableaushort{132,13}\text{ }\right)}
e_{\tiny
\ytableaushort{132,13}}
	+\theta_{\left(\text{ }
\tiny
\ytableaushort{113,22}
\text{ , }
\ytableaushort{132,23}\text{ }\right)}
e_{\tiny
\ytableaushort{132,23}}\\\\
	&\quad+\theta_{\left(\text{ }
\tiny
\ytableaushort{113,22}
\text{ , }
\ytableaushort{231,12}\text{ }\right)}
e_{\tiny
\ytableaushort{231,12}}
	+\theta_{\left(\text{ }
\tiny
\ytableaushort{113,22}
\text{ , }
\ytableaushort{231,13}\text{ }\right)}
e_{\tiny
\ytableaushort{231,13}}
	+\theta_{\left(\text{ }
\tiny
\ytableaushort{113,22}
\text{ , }
\ytableaushort{231,23}\text{ }\right)}
e_{\tiny
\ytableaushort{231,23}}\ 
	\\\\&=
	\theta_{11}\theta_{21}\theta_{12}\theta_{22}\theta_{33}e_3
	-\theta_{11}\theta_{21}\theta_{12}\theta_{23}\theta_{33}e_2
	-\theta_{11}\theta_{12}\theta_{22}\theta_{23}\theta_{33}e_1
	-\theta_{11}\theta_{21}\theta_{22}\theta_{32}\theta_{13}e_3
	\\&\quad+
	\theta_{11}\theta_{21}\theta_{32}\theta_{13}\theta_{23}e_2
	-\theta_{11}\theta_{22}\theta_{32}\theta_{13}\theta_{23}e_1
	-\theta_{21}\theta_{31}\theta_{12}\theta_{22}\theta_{13}e_3
	-\theta_{21}\theta_{31}\theta_{12}\theta_{13}\theta_{23}e_2
	\\&\quad+
	\theta_{31}\theta_{12}\theta_{22}\theta_{13}\theta_{23}e_1.
	\end{split}
	\end{equation}

To use \Cref{sec5_lemm_monomial-expansion-tableaux-2} in the proof of \Cref{sec5_prop_leading-term} it is necessary to identify the leading monomial $\theta^{\gamma_A} e^{\lambda_A'}$ with a term of the form $\theta_{AC}e_C$ with $C\in\mathbb{T}_A$. 
To do so we let $D_\lambda$, for $\lambda\in\mathbb{P}$, denote the Young tableau of shape $\lambda$ which has $1$'s in all entries of the first column, $2$'s in all entries of the second column and $l$'s in all entries of the $l$'th column. That is 
\begin{equation}
\label{sec5_eq_leading-tableau}
D_\lambda(k,l):=l,
	\end{equation}
for all $(k,l)\in \lambda$. 
From this definition it follows that if $A\in\mathbb{Y}$ is a s.s.\ Young tableau of shape $\lambda$ with $\ell(\lambda')\leq p$, then $D_\lambda$ is an $A$-restricted Young tableau in $\mathbb{T}_A$ and 
\begin{equation}
\theta_{AD_\lambda}e_{D_\lambda}= \theta^{\gamma_A}e^{\lambda_A'}.
	\end{equation}

To illustrate the tableau $D_\lambda$ and its relation to the leading monomial we continue the example of equation \eqref{sec5_eq_example-expansion-A-restricted}. Let $p=3$, $m=3$ and   
\begin{equation}
	\ytableausetup{centertableaux,boxsize=1.1em}
	A=\ytableaushort{113,22}.
\end{equation}
We write $\lambda=\lambda_A$. The shape of $A$ is then $\lambda=(3,2,0)$ which means that
\begin{equation}
\ytableausetup{centertableaux,boxsize=1.1em}
D_{\lambda} = 
\ytableaushort{123,12}.
	\end{equation}
From this we get 
\begin{equation}
\label{sec5_eq_leading-monomial-leading-tableau-ex1}
\theta_{AD_{\lambda}}e_{D_{\lambda}} 
	= \theta_{11}\theta_{12}\theta_{33}\theta_{21}\theta_{22}e_1e_2e_3e_1e_2
	= \theta_{11}\theta_{21}\theta_{12}\theta_{22}\theta_{33}e_3.
	\end{equation}
Note in addition that $\lambda'=(2,2,1)$ and that
\begin{equation}
\gamma_A=
\left(
\begin{matrix}
	1&1&0\\
	1&1&0\\
	0&0&1
	\end{matrix}
	\right).
	\end{equation}
The leading monomial of $A$ can then be calculated
\begin{equation}
\label{sec5_eq_leading-monomial-leading-tableau-ex2}
	\theta^{\gamma_A}e^{\lambda'}
	= \theta_{11}^1\theta_{21}^1\theta_{31}^0\theta_{12}^1\theta_{22}^1\theta_{32}^0\theta_{13}^0\theta_{23}^0\theta_{33}e_1^2e_2^2e_3
	= \theta_{11}\theta_{21}\theta_{12}\theta_{22}\theta_{33}e_3.
	\end{equation}
Comparing \eqref{sec5_eq_leading-monomial-leading-tableau-ex1} and \eqref{sec5_eq_leading-monomial-leading-tableau-ex2} we get $\theta_{AD_{\lambda}}e_{D_{\lambda}} = \theta^{\gamma_A}e^{\lambda'}$. 
In \eqref{sec5_eq_example-expansion-A-restricted} we calculated the monomial expansion of $\omega_A$. In accordance with \Cref{sec5_prop_leading-term} the leading term appears in that expansion with coefficient $1$.

\subsection{Proof of \Cref{sec5_prop_leading-term}}
\label{sec5.4}

\begin{proof}
Throughout this proof we let $A\in \mathbb{Y}$ be a s.s.\ Young tableau, with $\ell(\lambda_A')\leq p$. For ease of notation we write $D_\lambda=D_{\lambda_A}$.
We begin by noting that for any $C\in\mathbb{T}_A$ there exists a unique $\gamma\in M_{mp}(\Z_2)$ such that $\theta_{AC}e_C= \pm\theta^\gamma e^{\eta_\gamma}$, where the use of $\pm$ means that $\theta_{AC}e_C= \varepsilon\theta^\gamma e^{\eta_\gamma}$, for some $\varepsilon\in\{\pm1\}$. Together with \eqref{sec5_eq_monomial-expansion} this implies the first statement of \Cref{sec5_prop_leading-term}, namely that
\begin{equation}
\omega_A = 
	\sum_{\gamma\in M_{mp}(\Z_2)} \langle \theta^\gamma e^{\eta_\gamma}, \omega_A \rangle\theta^\gamma e^{\eta_\gamma}.
	\end{equation}

We now turn to proving that $\big\langle \theta^{\gamma_A} e^{\lambda_A'}, \omega_A \big\rangle = 1$. 
Since $\theta_{AD_\lambda}e_{D_\lambda}=\theta^{\gamma_A} e^{\lambda_A'}$, it follows from \Cref{sec5_lemm_monomial-expansion-tableaux-2} that we can prove this by showing that if $\theta_{AC}e_C=\pm\theta_{AD_\lambda}e_{D_\lambda}$, for some $C\in\mathbb{T}_A$, then $C=D_\lambda$. So suppose we have such a tableau $C$. Then
\begin{equation}
\label{sec5_eq_leading-term-proof-1}
\theta_{AD_\lambda}=\pm \theta_{AC}.
	\end{equation}
Using \eqref{sec5_eq_grassmann-monomial-tableaux} and \eqref{sec5_eq_leading-tableau} we can write
\begin{equation}
\label{sec5_eq_leading-term-proof-2}
\theta_{AC}=\pm\prod_{(k,l)\in\lambda} \theta_{A(k,l),C(k,l)}
\quad\text{ and }\quad
\theta_{AD_\lambda}=\pm\prod_{(k,l)\in\lambda} \theta_{A(k,l),l}.
	\end{equation}
Let $n=(\lambda_{A^1})_1$. Then by using \eqref{sec5_eq_leading-term-proof-1} and \eqref{sec5_eq_leading-term-proof-2} to compare the terms for which $A(k,l)=1$, that is those for which $(k,l)\in\lambda_{A^1}=\{(1,1),\dots,(1,n)\}$, we get the following identity:
\begin{equation}
\theta_{1,C(1,1)}\cdots \theta_{1,C(1,n)} = \pm \theta_{1,1}\cdots \theta_{1,n}.
	\end{equation}
By assumption $C$ is an $A$-restricted Young tableau, which means that $C(1,1)<\cdots<C(1,n)$ and thus $C(1,l)=l=D_\lambda(1,l)$, for $l\in\{1,\dots,n\}$. In other words, $C$ and $D_\lambda$ agree on all coordinates of $\lambda_{A^1}$.

Now suppose by induction that $C(k,l)=D_\lambda(k,l)$, for all $(k,l)\in\lambda_{A^t}$, that is for all $(k,l)\in\lambda$ with $A(k,l)\leq t$.
Since $C$ is an $A$-restricted Young tableau, this implies that
\begin{equation}
\label{sec5_eq_leading-term-proof-3}
(\lambda_{A^t})_k<C(k,(\lambda_{A^{t}})_k+1)<\cdots<C(k,(\lambda_{A^{t+1}})_k),
	\end{equation}
for all $k\in\{1,\dots,t+1\}$. 
By using \eqref{sec5_eq_leading-term-proof-1} and \eqref{sec5_eq_leading-term-proof-2} to compare the terms for which $A(k,l)=t+1$, that is those corresponding to the coordinates in
\begin{equation}
\big\{ (k,l): 1\leq k\leq t+1 \text{ and } (\lambda_{A^{t}})_k+1\leq l\leq (\lambda_{A^{t+1}})_k\big\},
	\end{equation}
we get the following identity
\begin{equation}
\label{sec5_eq_leading-term-proof-4}
\prod_{k=1}^{t+1} \theta_{t+1, C(k,(\lambda_{A^{t}})_k+1)}\cdots \theta_{t+1,C(k,(\lambda_{A^{t+1}})_k)}
	=
	\pm\prod_{k=1}^{t+1} \theta_{t+1, (\lambda_{A^{t}})_k+1}\cdots \theta_{t+1,(\lambda_{A^{t+1}})_k}.
	\end{equation}

By definition $A$ is a s.s.\ Young tableau. This means in particular that 
\begin{equation}
(\lambda_{A^{t+1}})_{k+1}\leq (\lambda_{A^{t}})_k\leq (\lambda_{A^{t+1}})_k,
	\end{equation}
for all $k\in\{1,\dots,t\}$.
With this in mind, the only way both \eqref{sec5_eq_leading-term-proof-3} and \eqref{sec5_eq_leading-term-proof-4} can be true is if $C(k,l)=l=D_\lambda(k,l)$, for all $k\in\{1,\dots,t+1\}$ and $l\in \{(\lambda_{A^{t}})_k+1,\dots, (\lambda_{A^{t+1}})_k\}$, that is for all $(k,l)$ with $A(k,l)=t+1$.
By assumption we already know that $C$ and $D_\lambda$ agree on the coordinates of $\lambda_{A^t}$, so we can conclude that they also agree on the coordinates of $\lambda_{A^{t+1}}$.  
This concludes the proof of the induction step, meaning that $C=D_\lambda$. From this we get $\big\langle \theta^{\gamma_A} e^{\lambda_A'}, \omega_A \big\rangle = 1$. 

We now turn to proving that $\big\langle \theta^{\gamma_A} e^{\lambda_A'}, \omega_B \big\rangle = 0$, for all $B\in \mathbb{Y}$ with $\ell(\lambda_B')\leq p$ and $A<B$. 
Let $C\in\mathbb{T}_B$. Then there exists $\gamma\in M_{mp}(\Z_2)$ such that $\theta_{BC}e_C= \pm\theta^\gamma e^{\eta_\gamma}$. Specifically,
\begin{equation}
\gamma_{ij}=\#\big\{ (k,l)\in\lambda: B(k,l)=i,\ C(k,l)=j\big\},
	\end{equation}
for all $i\in\{1,\dots,m\}$ and $j\in\{1,\dots,p\}$.
Note that 
\begin{equation}
\mu_A=
	\left(
	\sum_{j=1}^p (\gamma_A)_{1j},\dots,\sum_{j=1}^p (\gamma_A)_{mj}
	\right)
\quad\text{ and }\quad
\mu_B=
	\left(
	\sum_{j=1}^p \gamma_{1j},\dots,\sum_{j=1}^p \gamma_{mj}
	\right).
	\end{equation}
It then follows that $\gamma_A\neq\gamma$ and $\langle \theta^{\gamma_A} e^{\lambda_A'}, \theta_{BC}e_C\rangle=\langle \theta^{\gamma_A} e^{\lambda_A'}, \pm \theta^\gamma e^{\eta_\gamma}\rangle=0$ if $\mu_A\neq\mu_B$. Using \Cref{sec5_lemm_monomial-expansion-tableaux-2} we can thus conclude that $\big\langle \theta^{\gamma_A} e^{\lambda_A'}, \omega_B \big\rangle = 0$ if $\mu_A\neq\mu_B$. If on the other hand $\mu_A=\mu_B$, then the assumption that $A<B$ implies that there exists $s\in\{1,\dots,m\}$ such that 
\begin{equation}
\lambda_{A^i}=\lambda_{B^i}
\quad\text{ and }\quad
\lambda_{A^s}<\lambda_{B^s},
	\end{equation}
for all $i<s$, or equivalent that
\begin{equation}
\lambda_{A^i}'=\lambda_{B^i}'
\quad\text{ and }\quad
\lambda_{A^s}'>\lambda_{B^s}',
	\end{equation}
for all $i<s$. The statement $\lambda_{A^s}'>\lambda_{B^s}'$ implies that there exists $t\in\{1,\dots,\ell(\lambda_{A^s})\}$ such that 
\begin{equation}
(\lambda_{A^s}')_j=(\lambda_{B^s}')_j
\quad\text{ and }\quad
(\lambda_{A^s}')_t>(\lambda_{B^s}')_t,
	\end{equation}
for all $j<t$. 
With this we can make the following calculation
\begin{equation}
\label{sec5_eq_leading-term-proof-5}
\begin{split}
\sum_{i=1}^s\sum_{j=1}^t \gamma_{ij}
	&= 
	\#\big\{ (k,l)\in\lambda: 1\leq B(k,l)\leq s,\ 1\leq C(k,l)\leq t\big\}
	\leq
	\sum_{j=1}^t (\lambda_{B^s}')_j
	\\&<
	\sum_{j=1}^t (\lambda_{A^s}')_j
	=
	\sum_{i=1}^s\sum_{j=1}^t (\lambda_{A^i}')_j - (\lambda_{A^{i-1}}')_j
	\\&=
	\sum_{i=1}^s\sum_{j=1}^t (\gamma_A)_{ij},
	\end{split}
	\end{equation}
where the first inequality comes from noting that $C$ is $B$-restricted and thus row distinct, and the last identity follows from the observation that
\begin{equation}
(\gamma_A)_{i\alpha}=\#\{ \text{ $i$'s in the $\alpha$'th column of $A$ }\} = (\lambda_{A^i}')_\alpha - (\lambda_{A^{i-1}}')_\alpha,
	\end{equation}
for any $i\in\{1,\dots,m\}$ and $\alpha\in\{1,\dots,p\}$. Presently the inequality \eqref{sec5_eq_leading-term-proof-5} implies that $\gamma\neq \gamma_A$, which tells us that
\begin{equation}
\big\langle \theta^{\gamma_A} e^{\lambda_A'}, \theta_{BC}e_C \big\rangle 
	= \pm \big\langle \theta^{\gamma_A} e^{\lambda_A'}, \theta^\gamma e^{\eta_\gamma} \big\rangle
	= 0.
	\end{equation} 
By use of \Cref{sec5_lemm_monomial-expansion-tableaux-2} it then follows that $\big\langle \theta^{\gamma_A} e^{\lambda_A'}, \omega_B \big\rangle=0$, concluding the proof of \Cref{sec5_prop_leading-term}.
	\end{proof}

\appendix
\section{Graded Lexicographic ordering of $\N_0^m$ and $\mathbb{P}$}
\label{appA}

The graded lexicographic ordering of $\N_0^m$ is a total ordering and is defined by the following relation. Given $\mu,\eta\in \N_0^m$ we write $\mu<\eta$ if $|\mu|:=\sum_{i=1}^m \mu_i < |\eta|:=\sum_{i=1}^m \eta_i$, or if $|\mu|=|\eta|$ and if $\mu_j<\eta_j$, where $j$ is the first index for which $\mu$ and $\eta$ differ. 
To illustrate this ordering, we present here the ordering of the elements in $\mu\in\N_0^3$ with $|\mu|\leq 3$.
\begin{equation}
\begin{split}
&
(0,0,0)<(0,0,1)<(0,1,0)<(1,0,0)<(0,0,2)<(0,1,1)<(0,2,0)<
\\&
(1,0,1)<(1,1,0)<(2,0,0)<(0,0,3)<(0,1,2)<(0,2,1)<(0,3,0)<
\\&
(1,0,2)<(1,1,1)<(1,2,0)<(2,0,1)<(2,1,0)<(3,0,0).
	\end{split}
	\end{equation}

The graded lexicographic ordering of $\mathbb{P}$ is a total ordering and is defined by the following relation. Given $\lambda,\kappa\in\mathbb{P}$ we write $\lambda<\kappa$ if $|\lambda|:=\sum_{i=1}^{\ell(\lambda)}\lambda_i<|\kappa|:=\sum_{i=1}^{\ell(\kappa)}\kappa_i$, or if $|\lambda|=|\kappa|$ and if $\lambda_j<\kappa_j$, where $j$ is the first index for which $\lambda$ and $\kappa$ differ.
To illustrate this ordering, we present here the ordering of the elements in $\lambda\in\mathbb{P}$ with $|\lambda|\leq 4$.
\begin{equation}
\begin{split}
&
(0,0,0,0)<(1,0,0,0)<(1,1,0,0)<(2,0,0,0)<(1,1,1,0)<(2,1,0,0)<
\\&
(3,0,0,0)<(1,1,1,1)<(2,1,1,0)<(2,2,0,0)<(3,1,0,0)<(4,0,0,0).
	\end{split}
	\end{equation}

\section*{Acknowledgements}
The authors were supported by the EOS Research Project 30889451.

\end{document}